\documentclass{article}
\usepackage{amsmath, amsthm, epsfig,amssymb, multicol, tikz}
\usepackage{mathtools}

\usepackage{listings}
\usepackage{tikz-cd}
\usetikzlibrary{quotes}
\usetikzlibrary{decorations}
\usepackage{esvect}

\usepackage{xcolor}
\usepackage{lipsum}
\usepackage{amsfonts,amsmath,amssymb,amsthm}
\usepackage[shortlabels]{enumitem}
\usepackage{graphicx}

\usepackage{url}
\usepackage{hyperref}
\usepackage[pagewise]{lineno}
\usepackage{todonotes}
\usepackage{setspace}
\usepackage{lineno}

\def\pgfdecoratedcontourdistance{0pt}

\pgfkeys{/pgf/decoration/contour distance/.code={%
    \pgfmathparse{#1}%
    \let\pgfdecoratedcontourdistance=\pgfmathresult}%
}

\pgfdeclaredecoration{contour lineto}{start}
{
    \state{start}[next state=draw, width=0pt]{
        \pgfpathmoveto{\pgfpoint{0pt}{\pgfdecoratedcontourdistance}}%
    }
    \state{draw}[next state=draw, width=\pgfdecoratedinputsegmentlength]{       
        \pgfmathparse{-\pgfdecoratedcontourdistance*cot(-\pgfdecoratedangletonextinputsegment/2+90)}%
        \let\shorten=\pgfmathresult%
        \pgfpathlineto{\pgfpoint{\pgfdecoratedinputsegmentlength+\shorten}{\pgfdecoratedcontourdistance}}%
    }
    \state{final}{
        \pgfpathlineto{\pgfpoint{\pgfdecoratedinputsegmentlength}{\pgfdecoratedcontourdistance}}%
    }   
}

\newtheorem{lemma}{Lemma}
\newtheorem{theorem}{Theorem}
\newtheorem{corollary}{Corollary}
\newtheorem{definition}{Definition}

\newtheorem{prop}{Proposition}

\newtheorem{defn}{Definition}
\newtheorem{remark}{Remark}
\newtheorem{eg}{Example}

\title{Tur\'an Density of $2$-edge-colored Bipartite Graphs with Application on  $\{2, 3\}$-Hypergraphs}
\author{
Shuliang Bai \thanks{Shing-Tung Yau Center, Southeastern University, Nanjing, 210018, China,
({\tt sbai@seu.edu.cn}).} \and 
Linyuan Lu
\thanks{University of South Carolina, Columbia, SC 29208,
({\tt lu@math.sc.edu}). This author was supported in part by NSF
grant DMS 1600811.}
}

\begin{document}
\maketitle
\begin{abstract}
We consider the Tur\'an problems of $2$-edge-colored graphs. 
A $2$-edge-colored graph $H=(V, E_r, E_b)$ is a triple consisting of the vertex set $V$, 
the set of red edges $E_r$ and the set of blue edges $E_b$ with $E_r$ and $E_b$ do not have to be disjoint. 
The Tur\'an density $\pi(H)$ of $H$ is defined to be $\lim_{n\to\infty} \max_{G_n}h_n(G_n)$, where $G_n$ is chosen among all possible $2$-edge-colored graphs on $n$ vertices containing no $H$ as a sub-graph and $h_n(G_n)=(|E_r(G)|+|E_b(G)|)/{n\choose 2}$ is the formula to measure the edge density of $G_n$.  We will determine the Tur\'an densities of all $2$-edge-colored bipartite graphs.  We also give an important application of our study on the Tur\'an problems of $\{2, 3\}$-hypergraphs. 
\end{abstract}


\section{Introduction}
Given a graph $H$, the Tur\'an problem asks for the maximum possible number of edges (denoted as $ex(H, n)$) in a graph $G$ on $n$ vertices without a copy of $H$ as a sub-graph. 
The Mantel's theorem \cite{Mental} states that any graph on $n$ vertices with no triangle contains at most  $\lfloor n^2/4\rfloor$ edges.   Tur\'an\cite{Turan} proved that the maximal number of edges in a $k$-clique free graph  on $n$ vertices is at most $(k-2)n^2/(2k-2)$. 
The famed Erd\H{o}s-Stone-Simonovits Theorem \cite{ESS1, ESS2} proved that the Tur\'an density of any graph $H$  is $\pi(H)=1-\frac{1}{\mathcal{X}(H)-1}$, where $\mathcal{X}(H)$ is the chromatic number of $H$. For hypergraphs the extremal problems are harder, see Keevash\cite{Kee} for a complete survey of some results and methods on uniform hypergraphs.  Although Tur\'an type problems for graphs and hypergraphs have been actively studied for decades,  there are only few results on non-uniform hypergraphs, see \cite{MZ, PHTZ, JLU} for related work. 
Motivated by the study of non-uniform Tur\'an problems\cite{LuBai}, in this paper we study a Tur\'an-type problem on the edge-colored graphs and show an application on Tur\'an problems of non-uniform hypergraphs of edge size $2$ or $3$.

A hypergraph $H=(V, E)$ consists of a vertex set $V$ and
an edge set $E\subseteq 2^V$. An $r$-uniform hypergraph is a hypergraph such that all its hyperedges have size $r$.
Given positive integers $k\geq r\geq 2$, and a set of colors $C$, with $|C|=k$,  
a $k$-edge-colored $r$-uniform hypergraph  $H$ (for short, $k$-colored $r$-graph) is an $r$-uniform hypergraph that allows $k$ different colors on each hyperedge. We  express $H$ as $H=(V, E_1, E_2, \ldots, E_k)$ where $E_i$ denotes the set of hyperedges colored by $i$th color in $C$, note $E_1, E_2, \ldots, E_k$ do not have to be disjoint. 
We say $H'$ is a sub-graph of $H$,  denoted by $H'\subseteq H$,
if $V(H') \subseteq V(H)$, $E_i(H')\subseteq E_i(H)$ for every $i$. 
Given a family of $k$-colored $r$-graphs $\mathcal H$,  we say $G$ is {\it $\mathcal{H}$-free} if
it doesn't contain any member of $\mathcal{H}$ as a sub-graph.
To measure the edge density of $G$ of size $n$, we use $h_n(G)$,  which is defined by 
$$ h_n(G):=\sum\limits_{i=1}^{k}\frac{|E_i(G)|}{\binom{n}{r}},$$
where $n=|V(G)|$. 
Then we define the Tur\'an density of $\mathcal H$ as
$$\pi(\mathcal{H}):=\lim_{n\to\infty} \pi_n(\mathcal{H}) =\lim_{n\to\infty}  \max_{G_n} h_n(G_n),$$
where the maximum is taken over all ${\mathcal H}$-free $k$-colored $r$-graphs $G_n$ on $n$ vertices.

By a simple average argument of Katona-Nemetz-Simonovits theorem\cite{KNS}, this limit always exists. 
\begin{theorem}G
For any fixed family $\mathcal H$ of $k$-colored $r$-graphs, $\pi({\mathcal H})$ is well-defined, i.e.  $\lim_{n\to\infty} \pi_n(\mathcal{H})$ exists. 
\end{theorem}

When ${\mathcal H}=\{H\}$, we simply write $\pi(\{H\})$ as $\pi(H)$.
Note that $\pi(\mathcal {H})$ agrees with the definition of 
$$\pi({\mathcal {H}})=\frac{ex(\mathcal{H}, n)}{{n\choose r}},$$
where $ex(\mathcal {H}, n)$ is the maximum number of hyperedges in an $n$-vertex $\mathcal{H}$-free $k$-colored $r$-graph.

In this paper, we let $k=2$.  A $2$-edge-colored graph is a simple graph (without loops) where each edge is colored either red or blue, or both. We call an edge a double-colored edge if it is colored  with both colors.  For short, we call the $2$-edge-colored graphs simply as $2$-colored graphs. 
A $2$-colored graph $H$ can be written 
as a triple $H=(V, E_r, E_b)$ where $V$ is the vertex set,
$E_r\subseteq \binom{V}{2}$ is the set of red edges and $E_b\subseteq \binom{V}{2}$ is the set of blue edges. 
Denote $|E_r|$ and $|E_b|$ as the size of each set, denote $H_r, H_b$ as the induced sub-graphs of $H$ generated by all the red edges and all the blue edges respectively. 
A graph can be considered as a special $2$-colored graph with only one color. 
We say  $H$ is  {\it proper} if there exists at least one edge in each class $E_r$ and $E_b$.
Throughout the paper, we consider the proper $2$-colored graphs.  
The results in this paper is finished in year 2018 and recently we were noticed that our study is similar but different to a Tur\'an problem on edge-colored graphs defined by Diwan and Mubayi\cite{diwanmubayi} in which the authors ask for the minimum $m$, such that the $2$-colored graph $G$,  if both its red and blue edges are at least $m + 1$,  contains a given $2$-colored graph $F$? 
What we do differently in this paper is the study of the  Tur\'an density define above for $2$-colored graphs. 

It is easy to see that $\pi(H)\geq 1$ for any proper $2$-colored graph $H$,  since we can take a complete graph with all edges a single color that does not contain a copy of $H$. 
%

\begin{definition}
A  $2$-colored graph  $H$ is called bipartite if $H$ does not contain an odd cycle of length $l\geq 3$ with all edges colored by the same color. 
\end{definition}

For a $2$-colored graph  $H$, we say $H$ is {\it degenerate} if $\pi(H)=1$.
Note that if $H$ is degenerate, then it must be bipartite. 
Otherwise, say $H_b=(V,E_b)$ is not a bipartite
graph, one may consider the union of the red complete graph and an extremal graph respect to
$H_b$, then the resulting graph is a $H$-free $2$-colored graph with edge density at least
$1 + \pi(H_b)>1$, a contradiction. 

In this paper, 
we will determine the Tur\'an densities of all $2$-colored bipartite graphs and characterize the $2$-colored graphs achieving these Tur\'an values. The notation $[n]$ is the set of $\{1, \ldots, n\}$. For convenience, we represent an edge $\{a, b\}$  by $ab$. 

\begin{definition}
Given two $k$-colored $r$-graphs $G$ and $H$, a {\it graph homomporhism}
is a map $f\colon V(G)\to V(H)$ which keeps the colored edges, that is, $f(e)\in E_i(H)$ whenever
$e\in E_i(G)$ for $i\in [k]$.  We say $G$ is {\it $H$-colorable} if there is a graph homomorphism
from $G$ to $H$.
\end{definition}

\begin{theorem}\label{allresultsfor2colored}
The Tur\'an densities of all bipartite $2$-colored graphs are in the set $\left\{1, \frac{4}{3}, \frac{3}{2}\right\}$. 
\begin{enumerate}
\item A $2$-colored graph $H$ is degenerate if and only if it is $T$-colorable,
where $T$ is the $2$-colored graph with vertices $[4]$ and red edges $\{12, 13, 34\}$, blue edges $ \{12, 23, 34\}$.
\item A $2$-colored graph $H$ satisfies  $\pi(H)=\frac{4}{3}$, then $H$ must be $H_8$-colorable but not $T$-colorable, where $H_8$ is the $2$-colored graph with vertices $[8]$ and red edges 
  $E_r(H_8)=\{12, 13 ,24,  34,  16, 37, 48 , 25, 35, 18, 46, 27\}$, 
blue edges
 $E_b(H_8)=\{56, 57, 68, 78, 26, 15, 47, 38, 35, 18, 46, 27\}.$  
\item A $2$-colored bipartite graph $H$ satisfies $\pi(H)=\frac{3}{2}$, then $H$ is not $H_8$-colorable.
\end{enumerate}
\end{theorem}


Our consideration on  $2$-colored graphs is motivated by the study of Tur\'an density of non-uniform hypergraphs, which was first introduced by Johnston and Lu \cite{JLU}, then studied by us \cite{LuBai}.  We refer a  non-uniform hypergraph $H$ as $R$-graph, where $R$ is the set of all the cardinalities of edges in $H$.  
A {\it degenerate} $R$-graph $H$ has the smallest Tur\'an density, $|R|-1$, where $|R|$ is the size of set $R$. 
For a history of degenerate extremal graph problems, see \cite{dege}. Let $r\geq 3$, for $r$-uniform hypergraphs the $r$-partite hypergraphs are degenerate and they generalize the bipartite graphs. 
An interesting problem is what the degenerate non-uniform hypergraph look like?
In \cite{LuBai}, we prove that except for the case $R\neq \{1, 2\}$, there always exist non-trivial degenerate $R$-graphs for any set $R$ of two distinct positive integers.
The degenerate $\{1, 3\}$-graphs are characterized in \cite{LuBai}, what about the the degenerate $\{2, 3\}$-graphs? 
In the last section of this paper, we will apply the  $2$-colored graphs to bound the Tur\'an density of some $\{2, 3\}$-graphs. 

The paper is organized as follows: in section \ref{lemmas}, we show some  lemmas on 
the $k$-colored $r$-uniform hypergraphs; 
in section \ref{turanofedgecoloredgraph}, we classify the Tur\'an densities of all $2$-colored bipartite graphs;  in section \ref{23},  we give an application of the Tur\'an density of $2$-colored graphs on $\{2, 3\}$-graphs. 

\section{Lemmas on k-colored r-graphs}\label{lemmas}
\subsection{Supersaturation and Blowing-up}
In this section, we give some definitions and lemmas related to the $k$-colored $r$-graphs for $k\geq r \geq 2$. These are natural 
generalizations from the  Tur\'an theory of graphs.
We first define the {\it blow-up} of a $k$-colored $r$-graph. 
\begin{defn}[\it Blow-up Families]
For any $k$-colored $r$-graph $H$ on $n$ vertices and positive integers $s_1, s_2, \ldots, s_n$, the {\it blow-up} of  $H$ is a new $k$-colored $r$-graph, denoted by $H(s_1, s_2, \ldots, s_n)=(V, \ E_1, \ldots, \ E_k)$, satisfying 
\begin{itemize}
\item$ V:= \bigsqcup^n_{i=1}V_i$, where $|V_i|=s_i$, 
\item $E_j=\bigcup_{F\in E_j(H)} \prod_{i\in F}V_i$, for each $j\in [k]$. 
\end{itemize}
When $s_1=s_2=\cdots=s_n=s$, we simply write it as $H(s)$. 
\end{defn}
\begin{lemma}[Supersaturation]
For any $k$-colored $r$-graph $H$ and $a>0$,  
then there are $b, n_0>0$ so that if $G$ is a $k$-colored $r$-graph on 
$n>n_0$ vertices with $h_n(G)> \pi(H)+ a$ then $G$ contains at least $b{n \choose v(H)}$ copies of $H$. 
\end{lemma}
\begin{proof}
Since we have $\lim_{n \to \infty} \pi_n(H)=\pi(H)$, there exists an $n_0>0$ so that if $t>n_0$
then $\pi_t(H)< \pi(H)+\frac{a}{r}$. Suppose $n>t$, and $G$ is a $k$-colored $r$-graph on 
$n$ vertices with $h_n(G)> \pi(H)+ a$. 
Let $T$ represent any $t$-set, then $G$ must contain at least $\frac{a}{2}{n \choose t}$ $t$-sets $T\subseteq V(G)$
satisfying $h_t(G[T])> (\pi(H)+\frac{a}{2})$. Otherwise, we would have 
\begin{align*}
\sum\limits_{T}h_t(G[T])&\leq {n \choose t}(\pi(H)+\frac{a}{2})+ \frac{a}{2}{n \choose t}\\
&=(\pi(H)+a){n \choose t}.
\end{align*}

But we also have 
\begin{align*}
{t\choose r}\sum\limits_{T}h_t(G[T]) &={n-r \choose t-r}{n\choose r}h_n(G)\\
& >{n-r \choose t-r}{n \choose r}(\pi(H)+ a)\\
&= (\pi(H)+ a){t\choose r}{n \choose t}.
\end{align*}
A contradiction. Since $t>n_0$, it follows that each of the $\frac{a}{2}{n \choose t}$ $t$-sets $T\subseteq V(G)$ satisfying $h_t(G[T])> (\pi(H)+\frac{a}{r})$ contains a copy of $H$, so the number of copies of $H$ in G is at least $\frac{a}{2}{n \choose t}/{n-v(H) \choose t-v(H)}=\frac{a}{2}{n \choose v(H)}/{t \choose v(H)}$.  Let $b=\frac{a}{2}/{t\choose v(H)}$, the result follows. 
\end{proof}

The `blow-up' does not change the Tur\'an density of $k$-colored $r$-graphs.  The following result and proof are natural generalization of results on uniform hypergraphs, see \cite{Kee}. 
\begin{lemma} \label{squeeze}
  For any $s>1$ and any $k$-colored $r$-graph $H$, $\pi(H(s))=\pi(H).$
\end{lemma}

\begin{proof}
First, since any $H$-free $r$-graph $G$ is also $H(s)$-free, we have $\pi(H)\leq \pi(H(s))$. 
We will show that for any $a>0$, $\pi(H(s))< \pi(H)+a$. 

By the supersaturation lemma, for any  $a>0$, there are $b, n_0>0$ so that if $G$ is a $k$-colored $r$-graph on 
$n>n_0$ vertices with $h_n(G)> \pi(H)+ a$ then $G$ contains at least $b{n \choose v(H)}$ copies of $H$.
Consider an auxiliary $v(H)$-graph $U$ on the same vertex set as
$G$ such that the edges of $U$ correspond to copies of $H$ in $G$. Note that $U$ contains at least 
$b{n \choose v(H)}$ edges. For any $S>0$, if $n$ is large enough we can find a copy $K$ of 
$K_{v(H)}^{v(H)}(S)$ in $U$. Note that $K$ is the complete $v(H)$-partite v(H)-graph with $S$ vertices
 in each part, then $\pi(K)=0$. 
Fix one such $K$ in $U$. Color each
edge of $K$ with one of the $v(H)!$ colors corresponding to the possible orderings with which
the vertices of $H$ are mapped into the parts of $K$. By Ramsey theory, one of
the color classes contains at least $S^v/v!$ edges. For large enough $S$ (such that $S^v/v!\geq s$)
it follows that $U$ contains a monochromatic copy of $K_{v(H)}^{v(H)}(s)$, which gives a copy of $H(s)$ in $G$. 
Thus $\pi(H(s))< \pi(H)+a$. 
\end{proof}

Note when we say $G$ is $H$-colorable, it is equivalent to say $G$ is a sub-graph of a blow-up of $H$. 
It is easy to prove the following lemmas.
\begin{lemma}\label{homomorphism}
Let $\mathcal H$ be a family of $k$-colored $r$-graphs. 
If $G$ is $H$-colorable for any $H\in {\mathcal H}$, then $\pi(G)\leq \pi({\mathcal H})$.
\end{lemma}

\begin{defn}
Given two $k$-colored $r$-graphs $G_1$ and $G_2$ with vertices set
$V_1$ and $V_2$, we define the product of $G_1$ and $G_2$, denoted by 
$G_1\times G_2=(V_1\times V_2,\ E_1, \ldots, E_k)$,
where for any $i\in[k]$,  
$$E_i= E_i(G_1)\times E_i(G_2)=\{e\times f  \ \vert \ e \in E_i(G_1), f \in E_i(G_2)\},   $$  
where  $e\times f$ is defined through the following way: denote $e=\{v_1, \ldots, v_r\}\in E_i(G_1)$, $f=\{u_1, \ldots, u_r\}\in E_i(G_2),$ then
$e\times f=\cup_{\sigma\in S_r}  \{(v_1, u_{\sigma(1)}),\ldots, (v_r, u_{\sigma(r)})\}$, 
where $\sigma=(\sigma(1), \cdots,\sigma(r))$ takes over all permutations of $[r]$.
\end{defn}

\begin{lemma}\label{productcolorable}
A $k$-colored $r$-graph $G$ is $G_1$ and $G_2$ colorable, then it's $(G_1\times G_2)$-colorable. 
\end{lemma}

\begin{proof}
There exist two graph homomorphisms 
$f_1: V(G)\mapsto V(G_1)$ and $f_2: V(G)\mapsto V(G_2)$ such that
for any edge $e=\{v_1, \ldots, v_r\}\in E(G)$ (wlog, suppose $e\in E_1(G)$),  
we have  $$f_1(e)=\{f_1(v_1),\ldots, f_1(v_r)\}\in E_1(G_1),$$
and $$ f_2(e)=\{f_2(v_1), \ldots, f_2(v_r)\}\in E_1(G_2).$$ 

Define a map $f:=f_1\times f_2$ from $V(G)$ to $V(G_1)\times V(G_2)$, 
such that $f(v)=(f_1(v), f_2(v))$ for any $v\in V(G)$. 
Then we have 
$$f(e)
=\{(f_1(v_1), f_2(v_1)), \ldots, (f_1(v_r), f_2(v_r))\}\in f_1(e)\times f_2(e)\subseteq E_1(G_1 \times G_2).$$ 
Thus the map $f$ is a graph homomorphism. 
Hence $G$ is  $(G_1\times G_2)$-colorable. 
\end{proof}



\subsection{Construction of 2-colored graphs}
 To compute the lower bound of $\pi(H)$, we need to construct a family of $H$-free $2$-colored graphs $G_n$ with $h_n(G_n)$ as large as possible.
Here are three useful constructions.
\begin{description}
\item[ $G_A$: ] A $2$-colored graph $G_A$ on $n$ vertices is generated by partitioning  the vertex set into two parts such that  $V(G_A)= X \cup Y$ and the red edges either meet two vertices in $X$ or meet one vertex in $X$ plus the other 
in $Y$, the blue edges meet one vertex in $X$ plus the other in $Y$. 
In other words, the red edges  $E_r(G_A)=\binom{X}{2} \cup  \binom{X}{1}\times \binom{Y}{1}$ and blue edges  $E_b(G_A)=\binom{X}{1}\times \binom{Y}{1}$. 
 Let $|V(G_A)|=n$, 
 $|X|=xn$ and $|Y|=(1-x)n$ for some real number $x\in (0, 1)$.
 We have
\begin{align*}
  h_n(G_A) &= \frac{{|X|\choose 2}+ 2{|X|\choose 1}{|Y|\choose 1}}{{n\choose 2}}\\
           &=4x-3x^2+o_n(1), 
\end{align*}
which reaches the maximum $\frac{4}{3}$ at $x=\frac{2}{3}$. 
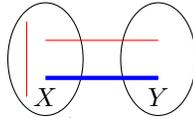
\begin{center}
  \begin{tikzpicture}[scale=0.25]
    \draw (0,0) ellipse (2 and 3);
    \draw (6,0) ellipse (2 and 3);
\draw[draw=red] (-1, -2) -- (-1, 2);
\draw[draw=red] ( 0, 1) -- (6, 1);
\draw[draw=blue, ultra thick] ( 0, -1) -- (6, -1);
\node at (0,-2) {$X$};
\node at (6,-2) {$Y$};
\node at (3,-4) {$G_A$: $h_n(G_A)=\frac{4}{3}+o_n(1)$  at $|X|=\frac{2}{3}n$.};
  \end{tikzpicture}
  \end{center}

\item [ $G_B$: ] It is obtained from $G_A$ by simply exchanging red edges with blue edges.  In other words,  the red edges $E_r(G_B)=\binom{X}{1} \times \binom{Y}{1}$ and blue edges $E_b(G_B)=\binom{X}{2} \cup  \binom{X}{1} \times \binom{Y}{1}$. 
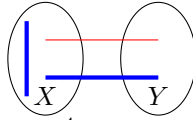
\begin{center}
 \begin{tikzpicture}[scale=0.25]
    \draw (0,0) ellipse (2 and 3);
    \draw (6,0) ellipse (2 and 3);
\draw[draw=blue, ultra thick] (-1, -2) -- (-1, 2);
\draw[draw=red] ( 0, 1) -- (6, 1);
\draw[draw=blue, ultra thick] ( 0, -1) -- (6, -1);
\node at (0,-2) {$X$};
\node at (6,-2) {$Y$};
\node at (3,-4) {$G_B$: $h_n(G_B)=\frac{4}{3}+o_n(1)$  at $|X|=\frac{2}{3}n$.};
  \end{tikzpicture}
\end{center}

\item [ $G_C$: ]
A $2$-colored graph $G_C$ on $n$ vertices is generated by partitioning the vertex set into two parts such that $V(G_C)= A \cup B$
and the red edges either meet two vertices in $A$ or meet one vertex in $A$ plus the other in $B$, the blue edges 
either meet two vertices in $B$ or meet one vertex in $A$ plus the other in $B$. 
In other words, the red edges 
$E_r(G_C)=\binom{A}{2} \cup  \binom{A}{1} \times \binom{B}{1}$ and blue edges  $E_b(G_C)=\binom{A}{1} \times  \binom{B}{1}\cup \binom{B}{2}$. 
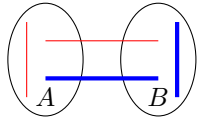
\begin{center}
  \begin{tikzpicture}[scale=0.25]
    \draw (0,0) ellipse (2 and 3);
    \draw (6,0) ellipse (2 and 3);
\draw[draw=red] (-1, -2) -- (-1, 2);
\draw[draw=red] ( 0, 1) -- (6, 1);
\draw[draw=blue, ultra thick] ( 0, -1) -- (6, -1);
\draw[draw=blue, ultra thick] ( 7, -2) -- (7, 2);
\node at (0,-2) {$A$};
\node at (6,-2) {$B$};
\node at (3,-5) {$G_C$: $h_n(G_C)=\frac{3}{2}+o_n(1)$  at $|A|=\frac{1}{2}n$.}; 
\end{tikzpicture}

\end{center}

\item [ $G_D$ and $G_E$: ]
Two variations of $G_C$ are the following constructions:
\begin{center}
\begin{tikzpicture}[scale=0.25]
    \draw (0,0) ellipse (2 and 3);
    \draw (6,0) ellipse (2 and 3);
\draw[draw=blue, ultra thick] (-1, -2) -- (-1, 2);
\draw[draw=red] ( 0, 1) -- (6, 1);
\draw[draw=blue, ultra thick] ( 0, -1) -- (6, -1);
\draw[draw=blue, ultra thick] ( 7, -2) -- (7, 2);
\node at (0,-2) {$X$};
\node at (6,-2) {$Y$};
\node at (3,-5) {$G_D$: $h_n(G_D)=\frac{3}{2}+o_n(1)$.}; 
\end{tikzpicture}
\hfil
\begin{tikzpicture}[scale=0.25]
    \draw (0,0) ellipse (2 and 3);
    \draw (6,0) ellipse (2 and 3);
\draw[draw=red] (-1, -2) -- (-1, 2);
\draw[draw=red] ( 0, 1) -- (6, 1);
\draw[draw=blue, ultra thick] ( 0, -1) -- (6, -1);
\draw[draw=red] ( 7, -2) -- (7, 2);
\node at (0,-2) {$C$};
\node at (6,-2) {$D$};
\node at (3,-5) {$G_E$: $h_n(G_E)=\frac{3}{2}+o_n(1)$.}; 
\end{tikzpicture}
\end{center}
\end{description}

\begin{eg}\label{GAGB=T}
The product of $G_A$ and $G_B$ is a blow-up of $T$, 
where $V(T)=[4]$, the red edges $\{12,13,34\}$ and the blue edges $\{12,23,34\}$:
\begin{center}
\begin{tikzpicture} [scale=0.6,
    every edge quotes/.append style={font=\scriptsize\sffamily, midway, auto, inner sep=1pt,  double},
    vertex/.style={circle, draw=black, fill=white, scale=0.6}
  ]
  \foreach \j/\i in {(0, 2)/1,(2, 2)/3,(2,0)/2,(0, 0)/4}
  \node [vertex] (\i) at \j {\i};
  \draw  [color=red] (1) edge (3);
  \draw  [color=blue, ultra thick] (2) edge (3);
  
\draw [
    postaction={
        decoration={contour lineto},
        draw=blue, ultra thick},
    postaction={
        decoration={contour lineto, contour distance=1.5pt}, 
        draw=red, decorate}
   ] 
   (1) -- (2);
   \draw [
    postaction={
        decoration={contour lineto},
        draw=blue, ultra thick},
    postaction={
        decoration={contour lineto, contour distance=1.5pt}, 
        draw=red, decorate}
   ] 
   (3) -- (4);
    \node at (1, -1) {$T$ };
  \end{tikzpicture}
\end{center}
\end{eg}
We define a map $f:V(H)\to \{1, 2, 3,4 \}$ as follows:
\begin{enumerate}
\item If $v$ appears in $X$ of $G_A$ and in $Y$ of $G_B$, set $f(v)=1$.
\item If $v$ appears in $Y$ of $G_A$ and in $X$ of $G_B$, set $f(v)=2$.
\item If $v$ appears in $X$ of $G_A$ and in $X$ of $G_B$, set $f(v)=3$.
\item If $v$ appears in $Y$ of $G_A$ and in $Y$ of $G_B$, set $f(v)=4$.
\end{enumerate}
One can check $f$ is a graph homomorphism from the product $G_A\times G_B$ to $T$.

\section{Tur\'an density of bipartite 2-colored graphs}\label{turanofedgecoloredgraph}
In this section, we will prove results in Theorem \ref{allresultsfor2colored}. 
We first give a boundary to divide the Tur\'an densities of $2$-colored non-bipartite graphs 
and  $2$-colored bipartite graphs. 

\begin{lemma}\label{3/2}

\begin{enumerate}
\item For any $2$-colored non-bipartite graph $H$, $\pi(H)\geq\frac{3}{2}$. 
\item For any $2$-colored bipartite graph $H$, $\pi(H)\leq\frac{3}{2}$. 
\end{enumerate}
\end{lemma}
Before proceeding to the proof, let us see several important $2$-colored graphs whose Tur\'an density achieve the value $\frac{3}{2}$. 
\begin{lemma}\label{turanofdoubletriangle}
Let $K_3$ be a triangle with three double-colored edges, i.e. $K_3=([3], \{12, 13, 23\},  \{12, 13, 23\})$.  Then 
$$ex(n, K_3)={n\choose 2} +\left\lfloor \frac{n^2}{4}\right \rfloor.$$
 In particular, $\pi(K_3)=\frac{3}{2}$. 
\end{lemma}

\begin{proof}
Observe that $K_3$ is not contained in $G_C$, thus $\pi(K_3)\geq \frac{3}{2}$. 
Now we prove the other direction. 
Let $n$ be a positive integer and $G$ be any $K_3$-free $2$-colored graph on $n$ vertices. Construct an auxillary graph $F$ on the same vertex set $V(G)$ and with the edge sets consisting of all double-colored edges in $G$. Let $H=E_r(F)$ consisting of all red colored edges of $F$.  Notice that $H$ is triangle-free. 
By Mantel theorem, we have 
$$|E(H)|\leq \left\lfloor \frac{n^2}{4} \right\rfloor.$$ 
Note that $H$ is a subgraph of $G$ and the number of the rest of edges in $G$ is at most ${n\choose 2}$. 
Therefore, we have
$$|E(G)|\leq {n\choose 2}+ |E(H)|\leq {n\choose 2}+ \left\lfloor \frac{n^2}{4}\right\rfloor=\left(\frac{3}{2}+o(1)\right){n\choose 2}.$$
This implies that $\pi(K_3)=\frac{3}{2}$.
\end{proof}

\begin{corollary}
Let $K_3^{-}=([3], \{12, 13, 23\},\{12, 13\})$, then $\pi(K_3^{-})=\frac{3}{2}$.
\end{corollary}

\begin{proof}
Since $K_3^{-}$ is a sub-graph of $K_3$, then $\pi(K_3^{-})\leq \frac{3}{2}$. 
By Lemma \ref {3/2},  $\pi(K_3^{-})\geq \frac{3}{2}$.  The result follows. 
\end{proof}

Except the $2$-colored non-bipartite graph, some bipartite graphs also achieves $\pi(H)=\frac{3}{2}$. See
the following $2$-colored graph on four vertices $\{1, 2, 3, 4\}$:
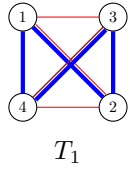
\begin{center}
\begin{tikzpicture} [scale=0.6,
    every edge quotes/.append style={font=\scriptsize\sffamily, midway, auto, inner sep=1pt,  double},
    vertex/.style={circle, draw=black, fill=white, scale=0.6}
  ]
  \foreach \j/\i in {(0, 2)/1,(2, 2)/3,(2,0)/2,(0, 0)/4}
  \node [vertex] (\i) at \j {\i};
  \draw  [color=blue,  ultra thick] (1) edge (4);
  \draw  [color=blue,   ultra thick] (2) edge (3);
  \draw  [color=red] (1) edge  (3);
  \draw  [color=red] (2) edge  (4);
  
\draw [
    postaction={
        decoration={contour lineto},
        draw=blue,  ultra thick},
    postaction={
        decoration={contour lineto, contour distance=1.5pt}, 
        draw=red, decorate}
   ] 
   (1)--(2);
   \draw [
    postaction={
        decoration={contour lineto},
        draw=red},
    postaction={
        decoration={contour lineto, contour distance=1.5pt}, 
        draw=blue, decorate, ultra thick}
   ] 
   (3)--(4);
  \node at (1, -1) {$T_1$};
  \end{tikzpicture}
\end{center}

\begin{lemma}\label{Pi(T_1)}
 $T_1=([4], \{12, 34, 13, 24\}, \{12, 34, 14, 23\}$. Then
 $$ex(n, T_1)={n\choose 2} +\left\lfloor \frac{n^2}{4}\right \rfloor \mbox{ for any } n\not= 3$$
 and $ex(3, T_n)=6$.
In particular, we have
 $\pi(T_1)=\frac{3}{2}$. 
\end{lemma}

\begin{proof} 
When $n\leq 3$, the complete $2$-colored graph does not contain $T_1$.
Thus $ex(n, T_1)=0,0,2,6$ when $n=0,1,2,3$, respectively. The assertion holds for $n\leq 3$. It is sufficient to prove for $n\geq 4$.
Since $T_1$ is not contained in $G_C$, 
we have $$ ex(n, T_1)\geq {n\choose 2} +\left\lfloor \frac{n^2}{4}\right \rfloor.$$
Now we prove the other direction
by induction. We may assume $n\geq 4$.
Let $n$ be a positive integer and $G$ be any $T_1$-free $2$-colored graph on $n$ vertices.

\begin{description}
\item[Case 1:] $G$ doesn't contain $K_3$ as a subgraph,
by Lemma \ref{turanofdoubletriangle}, 
we have 
$$|E(G)|\leq ex(n, K_3)= {n\choose 2} +\left\lfloor \frac{n^2}{4}\right \rfloor.$$


\item[Case 2:]
$G$ contains a copy of $K_3$,
let $V_1=\{a, b, c\}$ be the vertices of this 
triangle and $V_2=V(G)\setminus V_1$. Then there are at most $4$ edges from any vertex in $V_2$ to $V_1$.
To see this, suppose there are $5$ edges from the vertex $w\in V_2$ to $V_1$, 
then there are only two possible graphs on $V_1 \cup \{w\}$ and each of them
contains a copy of $T_1$. A contradiction.

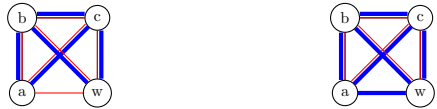
\begin{center}
 \begin{tikzpicture} [
    every edge quotes/.append style={font=\scriptsize\sffamily, midway, auto, inner sep=1pt,  double},
    vertex/.style={circle, draw=black, fill=white, scale=0.6}
  ]
  \foreach \j/\i in {(0, 0)/a,(0, 1)/b,(1, 1)/c,(1, 0)/w}
 \node [vertex] (\i) at \j {\i};
  \draw  [color=red] (a) edge (w);
  
  \draw [
    postaction={
        decoration={contour lineto},
        draw=red},
    postaction={
        decoration={contour lineto, contour distance=1.5pt}, 
        draw=blue, decorate,  ultra thick}
   ] 
   (a)--(b);
   \draw [
    postaction={
        decoration={contour lineto},
        draw=blue, ultra thick},
    postaction={
        decoration={contour lineto, contour distance=1.5pt}, 
        draw=red, decorate}
   ] 
   (a)--(c);

\draw [
    postaction={
        decoration={contour lineto},
        draw=red},
    postaction={
        decoration={contour lineto, contour distance=1.5pt}, 
        draw=blue, decorate,  ultra thick}
   ] 
   (b)--(c);
   \draw [
    postaction={
        decoration={contour lineto},
        draw=blue, ultra thick},
    postaction={
        decoration={contour lineto, contour distance=1.5pt}, 
        draw=red, decorate}
   ] 
   (b)--(w);
   \draw [
    postaction={
        decoration={contour lineto},
        draw=red},
    postaction={
        decoration={contour lineto, contour distance=1.5pt}, 
        draw=blue, decorate,  ultra thick}
   ] 
   (c)--(w);
 \end{tikzpicture}  
 \hfil
  \begin{tikzpicture} [
    every edge quotes/.append style={font=\scriptsize\sffamily, midway, auto, inner sep=1pt,  double},
    vertex/.style={circle, draw=black, fill=white, scale=0.6}
  ]
  \foreach \j/\i in {(0, 0)/a,(0, 1)/b,(1, 1)/c,(1, 0)/w}
  \node [vertex] (\i) at \j {\i};
  \draw  [color=blue,  ultra thick] (a) edge (w);
 \draw [
    postaction={
        decoration={contour lineto},
        draw=red},
    postaction={
        decoration={contour lineto, contour distance=1.5pt}, 
        draw=blue, decorate,  ultra thick}
   ] 
   (a)--(b);
   \draw [
    postaction={
        decoration={contour lineto},
        draw=blue, ultra thick},
    postaction={
        decoration={contour lineto, contour distance=1.5pt}, 
        draw=red, decorate}
   ] 
   (a)--(c);

\draw [
    postaction={
        decoration={contour lineto},
        draw=red},
    postaction={
        decoration={contour lineto, contour distance=1.5pt}, 
        draw=blue, decorate,  ultra thick}
   ] 
   (b)--(c);
   \draw [
    postaction={
        decoration={contour lineto},
        draw=blue, ultra thick},
    postaction={
        decoration={contour lineto, contour distance=1.5pt}, 
        draw=red, decorate}
   ] 
   (b)--(w);
   \draw [
    postaction={
        decoration={contour lineto},
        draw=red},
    postaction={
        decoration={contour lineto, contour distance=1.5pt}, 
        draw=blue, decorate,  ultra thick}
   ] 
   (c)--(w);
  \end{tikzpicture}
 
\end{center}
\end{description}

Applying the inductive hypothesis to $G[V_2]$, we have
 $$|E(G[V_2])|\leq {n-3\choose 2} +\left\lfloor \frac{(n-3)^2}{4}\right \rfloor +\epsilon.$$ Here $\epsilon=1$ if $n=6$ and 0 otherwise.
 
Then the number of edges in $G$ is:

if $n\neq 6$, 
\begin{align*}
|E(G)|
&= |E(G[V_1])| + |E(G[V_2])| +|E(V_1, V_2)|\\
&\leq 6 + {n-3\choose 2} +\left\lfloor \frac{(n-3)^2}{4}\right \rfloor  + 4 (n-3)\\
&= {n\choose 2} + n+ \left\lfloor \frac{n^2-6n +9}{4}\right \rfloor \\
&=  {n\choose 2} +  \left\lfloor \frac{n^2-2n +9}{4}\right \rfloor \\
&= {n\choose 2} +  \left\lfloor \frac{n^2}{4}\right \rfloor. 
\end{align*}
if $n=6$, 
\begin{align*}
|E(G)|
&= |E(G[V_1])| + |E(G[V_2])| +|E(V_1, V_2)|\\
&\leq 6 + {n-3\choose 2} +\left\lfloor \frac{(n-3)^2}{4}\right \rfloor +\epsilon + 4 (n-3)\\
&=24\\
&={6\choose 2} +  \left\lfloor \frac{6^2}{4}\right \rfloor. 
\end{align*}
The induction step is finished. 
It follows that $h_n(G)\leq \frac{3}{2}$. 
Therefore, 
$\pi(T_1)=\frac{3}{2}$.
\end{proof}

\begin{proof}[Proof of Lemma \ref{3/2}]
For Item 1, let $H$ be a $2$-colored non-bipartite graph, without loss of generality, assume $H$ contains an odd 
cycle with red edges. For any $n$, let $G$ be a $2$-colored graph generated by construction $G_D$, then $H$ can not be contained in $G$. Similarly,  $H$ contains an odd 
cycle with blue edges, then it is not contained in any $2$-colored graph generated by construction $G_E$. 
Thus $\pi(H)\geq \frac{3}{2}$. 

For Item 2, it is sufficient to prove that 
any $2$-colored bipartite graph $H$ is $T_1$-colorable. 
For any $2$-colored bipartite graph $H$, 
the sub-graph $H_r$ can be partitioned into two disjoint parts $V_1(H_r)$ and $V_2(H_r)$ such that the red edges form a bipartite graph between $V_1(H_r)$ and $V_2(H_r)$. Similarly for the sub-graph $H_b$, the blue edges form 
a bipartite graph between $V_1(H_b)$ and $V_2(H_b)$.  Let $S$ be the set of vertices incidents to double colored edges, then $S$ can be divided into 
four classes: $V_1(H_r) \cap V_1(H_b), V_1(H_r) \cap V_2(H_b), V_2(H_r) \cap V_1(H_b)$ and $V_2(H_r) \cap V_2(H_b)$. 
We define a map $f:V(H)\to \{1, 2, 3,4 \}$ as follows: 
\begin{enumerate}
\item If $v\in V_1(H_r) \cap V_1(H_b)$, set $f(v)=1$.
\item If $v\in V_1(H_r) \cap V_2(H_b)$, set $f(v)=4$.
\item If $v\in V_2(H_r) \cap V_1(H_b)$, set $f(v)=3$.
\item If $v\in V_2(H_r) \cap V_2(H_b)$, set $f(v)=2$.
\item If $uv\in E_r(H)\setminus E_b(H)$, set $f(u)=1, f(v)=2$.
\item If $uv\in E_b(H)\setminus E_r(H)$, set $f(u)=3, f(v)=4$.
\end{enumerate}
One can verify that this map $f$ is a graph homomorphism from $H$ to $T_1$.
Hence,  $\pi(H)\leq \frac{3}{2}$. 
\end{proof}

\subsection{The degenerate 2-colored graphs}
In this part, we will determine the degenerate $2$-colored graphs.
Recall the Example \ref{GAGB=T} in Section 2, we will show that the $2$-colored bipartite graph $T=([4], \{12,13,34\}, \{12,23,34\})$ plays an important role. 

\begin{lemma} \label{Tisdege}
Let $n$ be a positive integer, for any $T$-free $2$-colored graph $G$ on $n$ vertices, $G$ has at most
$\binom{n+1}{2}$ edges. Thus  $T$ is degenerate.
\end{lemma}
\begin{proof}
We will prove this lemma by induction on $n$. It is trivial for $n=1,2,3, 4$. Assume $n\geq 5$.
We assume that the statement holds for any $T$-free $2$-colored graphs on less than $n$ vertices. 

Let $G=(V, E_r, E_b)$ be a $T$-free $2$-colored graph on $n$ vertices. 
We also assume $G$ contains at least one double-colored edge $uv$, or else 
$|E_r(G)|+|E_b(G)|\leq \binom{n}{2}< \binom{n}{2} <\binom{n+1}{2}$. 
Then $G$ is one of the following cases.
\begin{description}
\item[Case 1:] There exists a vertex $w$ so that both $uw$ and $vw$ are double-colored edges.
Since $G$ is $T$-free, there is no double-colored edges from $u,v,w$ to the rest of the vertices.
By inductive hypothesis, when $G$ is restricted to the complement set of $\{u,v,w\}$, the number of edges of $G[V\setminus \{u,v,w\}]$ is at most
$\binom{n-2}{2}$ edges. Thus, $G$ has at most
$$6+3(n-3)+\binom{n-2}{2}=\binom{n+1}{2}.$$
\item[Case 2:] Now we assume no such $w$ exists. 
Let $X=\{x\in V\colon |E(\{x\},\{u,v\})|\geq 3\}$. 
That is,  for each vertex $x\in X$, $x$ has exactly 3 edges connecting to $u$ and $v$. 
Since $G$ is $T$-free, for each $x\in X$, $x$ has no double-colored edges to any vertex not in $\{u,v,x\}$. In particular, the induced sub-graph $G[X]$ of $G$ has no double-colored edge. 
Let $V_1=\{u,v\}\cup X$ and $V_2$ be the complement set. 
Then the induced sub-graph $G[V_1]$ has at most $$2+3|X|+\binom{|X|}{2}<\binom{|X|+3}{2}=\binom{|V_1|+1}{2}$$
edges.
Applying the inductive hypothesis to $G[V_2]$, then $G[V_2]$ has at most $\binom{|V_2|+1}{2}$ edges. Note that all edges from $X$ to $V_2$ are single colored and the number of edges from $\{u,v\}$ 
to each vertex in $V_2$ is at most $2$. Thus the total number of edges from $V_1$ to $V_2$
is at most $|V_1||V_2|$ edges. Combining these facts together,  we have $G$ has at most $N$ edges, where
$$N=\binom{|V_1|+1}{2}+ |V_1||V_2| + \binom{|V_2|+1}{2}=\binom{|V|+1}{2}. $$
\end{description}
We finish the inductive step. 
Then we have $$\pi(T)=\lim_{n\to\infty} \max_{G_n} h_n(G_n)\leq 
 \lim_{n\to\infty}  \frac{\binom{n+1}{2}}{\binom{n}{2}}= 1, $$ implying $\pi(T)=1$.  $T$ is degenerate.
\end{proof}


\begin{proof}[Proof of Item 1 of Theorem \ref{allresultsfor2colored}]
Assume $H$ is a degenerate $2$-colored graph, then it must be $G_A$ and $G_B$-colorable.
By Lemma \ref{productcolorable}, it must be $G_A\times G_B$-colorable. 
Note that the product of this two graphs is $T$-colorable.  
Thus $H$ is $T$-colorable.  By Lemma \ref{Tisdege},   the result follows. 
\end{proof}
\begin{remark}\label{4/3then1}
Above proof implies that for any $2$-colored graph $H$ with 
$\pi(H)<\frac{4}{3}$, $\pi(H)=1$.  
\end{remark}

\subsection{Non-degenerate 2-colored bipartite graphs}
In this part, we will further classify the non-degenerate $2$-colored bipartite graphs. 
By Lemma \ref{3/2}, the largest possible Tur\'an density of a  $2$-colored bipartite graph $H$ is $\frac{3}{2}$, 
so if $\pi(H)< \frac{3}{2}$, it must be contained in the construction $G_c$ and its variations $G_D, G_E$, thus it must be colored by the product of these constructions.  
While the product of graphs generated by the three constructions is a blow-up of following graph $H_8$. Let $ACX$ stand for the vertex in $A\times C\times X$, similar for other labels:
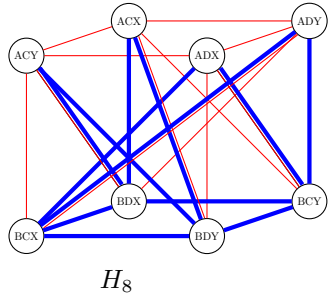
\begin{center}
\begin{tikzpicture} [scale=0.6,
    every edge quotes/.append style={font=\scriptsize\sffamily, midway, auto, inner sep=1pt,  double},
    vertex/.style={circle, draw=black, fill=white, scale=0.6}
  ]
  \foreach \j/\i in {(-2,1,-1)/ACX,(2,1,-1)/ADY,(-3.5,1,1)/ACY,(0.5,1,1)/ADX,(-2,-3,-1)/BDX,(2,-3,-1)/BCY,(-3.5,-3,1)/BCX,(0.5,-3,1)/BDY}
  \node [vertex, scale=0.7] (\i) at \j {\i};
  \draw  [color=red] (ACX) edge (ADY);
  \draw  [color=red] (ACX) edge (ACY);

  \draw [color=red] (ACX) edge (BCY);
  \draw [color=red] (ADY) edge (ADX);
  \draw [color=red] (ADY) edge (BDX);
  \draw [color=red] (ACY) edge (BCX);
  \draw [color=red] (ADX) edge (BDY);
   \draw [color=red] (ACY) edge (ADX);
  
  \draw  [color=blue, ultra thick] (ACX) edge  (BDX);
  \draw  [color=blue, ultra thick] (ADY) edge  (BCY);
  \draw  [color=blue,  ultra thick] (ACY) edge  (BDY);
  \draw  [color=blue,  ultra thick] (ADX) edge  (BCX);
   \draw  [color=blue, ultra thick] (BDX) edge  (BCY);
  \draw  [color=blue,  ultra thick] (BDX) edge  (BCX);
   \draw  [color=blue,  ultra thick] (BDY) edge  (BCX);
  \draw  [color=blue, ultra thick] (BDY) edge  (BCY);

  \draw [
    postaction={
        decoration={contour lineto},
        draw=red},
    postaction={
        decoration={contour lineto, contour distance=1.5pt}, 
        draw=blue, decorate,  ultra thick}
   ] 
   (ACY) -- (BDX);
   \draw [
    postaction={
        decoration={contour lineto},
        draw=blue, ultra thick},
    postaction={
        decoration={contour lineto, contour distance=1.5pt}, 
        draw=red, decorate}
   ] 
   (ACX) -- (BDY);

\draw [
    postaction={
        decoration={contour lineto},
        draw=red},
    postaction={
        decoration={contour lineto, contour distance=1.5pt}, 
        draw=blue, decorate,  ultra thick}
   ] 
   (ADX) -- (BCY);
   \draw [
    postaction={
        decoration={contour lineto},
        draw=blue, ultra thick},
    postaction={
        decoration={contour lineto, contour distance=1.5pt}, 
        draw=red, decorate}
   ] 
   (ADY) -- (BCX);
 \node at (-1.5, -4, 1) {$H_8$};  
 \end{tikzpicture}
\end{center}

To compute the Tur\'an density of $H_8$, we need the following $2$-colored graph 
$T_2=([4], \{12,14, 23, 24, 34\}, \{12, 13, 14, 23, 34\})$.   
$T_2$ is not contained in a variation of $G_C$, thus  $\pi(T_2)\geq \frac{3}{2}$.
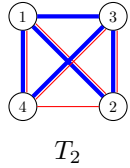
\begin{center}
 \begin{tikzpicture} [scale=0.6,
    every edge quotes/.append style={font=\scriptsize\sffamily, midway, auto, inner sep=1pt,  double},
    vertex/.style={circle, draw=black, fill=white, scale=0.6}
  ]
  \foreach \j/\i in {(-1, 2)/1,(1, 2)/3,(1,0)/2,(-1, 0)/4}
  \node [vertex] (\i) at \j {\i};
    
\draw[draw=red] (4) -- (2);
\draw[draw=blue, ultra thick] (1) -- (3);

\draw [
    postaction={
        decoration={contour lineto},
        draw=blue, ultra thick},
    postaction={
        decoration={contour lineto, contour distance=1.5pt}, 
        draw=red, decorate}
   ] 
   (1) -- (2);
   \draw [
    postaction={
        decoration={contour lineto},
        draw=blue, ultra thick},
    postaction={
        decoration={contour lineto, contour distance=1.5pt}, 
        draw=red, decorate}
   ] 
   (3) -- (4);
   \draw [
    postaction={
        decoration={contour lineto},
        draw=blue, ultra thick},
    postaction={
        decoration={contour lineto, contour distance=1.5pt}, 
        draw=red, decorate}
   ] 
   (1) -- (4);
   \draw [
    postaction={
        decoration={contour lineto},
        draw=blue, ultra thick},
    postaction={
        decoration={contour lineto, contour distance=1.5pt}, 
        draw=red, decorate}
   ] 
   (3) -- (2);

\node at (0,-1) {$T_2$};
  \end{tikzpicture}
\end{center}

\begin{lemma}\label{T1T2free}
For any positive integer $n$, let $G$ be a  $\{T_1, T_2\}$-free $2$-colored graph on $n$ vertices. Then 
$|E(G)|\leq  {n\choose 2} +\left\lfloor \frac{n^2}{6}\right\rfloor. $
Thus $\pi(\{T_1, T_2\})\leq \frac{4}{3}$.
\end{lemma}

\begin{proof}
It is not hard to check the cases for $n\leq 3$. 
Let $n\geq 4$, by induction on $n$ we assume the statement holds for any $\{T_1, T_2\}$-free graph
on less than $n$ vertices.
Note if $G$ contains no double-colored edge, the result is trivial.
Thus we assume $G$ contains at least one double-colored edge.
Then $G$ is one of the following cases.
 \begin{description}
\item[Case 1:]$G$ contains a triangle consisting of three double-colored edges,
let $V_1=\{a, b, c\}$ be the vertices of this 
triangle and $V_2=V(G)\setminus V_1$.
By Lemma \ref{Pi(T_1)} ``Case 2", 
for any vertex $w\in V_2$, 
there are at most $4$ edges form $w$ to $V_1$. 

\item[Case 2:] $G$ contains $V_1=\{a, b, c\}$ such that
$|E(G[V_1])|=5$. 
Let $V_2=V(G)\setminus V_1$. 
For any vertex $w\in V_2$,  there are at most $4$ edges to $V_1$. 
One can check the following graphs with $5$ edges from $w$ to $V_1$
contain $T_1$, $T_2$ and $T_1$ respectively. 
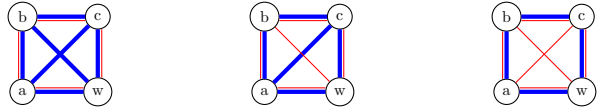
\begin{center}
  \begin{tikzpicture} [
    every edge quotes/.append style={font=\scriptsize\sffamily, midway, auto, inner sep=1pt,  double},
    vertex/.style={circle, draw=black, fill=white, scale=0.6}
  ]
  \foreach \j/\i in {(0, 0)/a,(0, 1)/b,(1, 1)/c,(1, 0)/w}
  \node [vertex] (\i) at \j {\i};
  \draw  [color=blue,  ultra thick] (a) edge (c);
  \draw  [color=blue,  ultra thick] (b) edge (w);
 \draw [
    postaction={
        decoration={contour lineto},
        draw=blue, ultra thick},
    postaction={
        decoration={contour lineto, contour distance=1.5pt}, 
        draw=red, decorate}
   ] 
   (a) -- (b);
   \draw [
    postaction={
        decoration={contour lineto},
        draw=blue, ultra thick},
    postaction={
        decoration={contour lineto, contour distance=1.5pt}, 
        draw=red, decorate}
   ] 
   (a) -- (w);
   \draw [
    postaction={
        decoration={contour lineto},
        draw=blue, ultra thick},
    postaction={
        decoration={contour lineto, contour distance=1.5pt}, 
        draw=red, decorate}
   ] 
   (c) -- (b);
   \draw [
    postaction={
        decoration={contour lineto},
        draw=blue, ultra thick},
    postaction={
        decoration={contour lineto, contour distance=1.5pt}, 
        draw=red, decorate}
   ] 
   (c) -- (w);
  \end{tikzpicture}
  \hfil
  \begin{tikzpicture} [
    every edge quotes/.append style={font=\scriptsize\sffamily, midway, auto, inner sep=1pt,  double},
    vertex/.style={circle, draw=black, fill=white, scale=0.6}
  ]
  \foreach \j/\i in {(0, 0)/a,(0, 1)/b,(1, 1)/c,(1, 0)/w}
  \node [vertex] (\i) at \j {\i};
  \draw  [color=blue,  ultra thick] (a) edge (c);
  \draw  [color=red] (b) edge (w);
 \draw [
    postaction={
        decoration={contour lineto},
        draw=blue, ultra thick},
    postaction={
        decoration={contour lineto, contour distance=1.5pt}, 
        draw=red, decorate}
   ] 
   (a) -- (b);
   \draw [
    postaction={
        decoration={contour lineto},
        draw=blue, ultra thick},
    postaction={
        decoration={contour lineto, contour distance=1.5pt}, 
        draw=red, decorate}
   ] 
   (a) -- (w);
   \draw [
    postaction={
        decoration={contour lineto},
        draw=blue, ultra thick},
    postaction={
        decoration={contour lineto, contour distance=1.5pt}, 
        draw=red, decorate}
   ] 
   (c) -- (b);
   \draw [
    postaction={
        decoration={contour lineto},
        draw=blue, ultra thick},
    postaction={
        decoration={contour lineto, contour distance=1.5pt}, 
        draw=red, decorate}
   ] 
   (c) -- (w);
  \end{tikzpicture}
  \hfil
\begin{tikzpicture} [
    every edge quotes/.append style={font=\scriptsize\sffamily, midway, auto, inner sep=1pt,  double},
    vertex/.style={circle, draw=black, fill=white, scale=0.6}
  ]
  \foreach \j/\i in {(0, 0)/a,(0, 1)/b,(1, 1)/c,(1, 0)/w}
  \node [vertex] (\i) at \j {\i};
 \draw  [color=red] (a) edge (c);
  \draw  [color=red] (b) edge (w);
 \draw [
    postaction={
        decoration={contour lineto},
        draw=blue, ultra thick},
    postaction={
        decoration={contour lineto, contour distance=1.5pt}, 
        draw=red, decorate}
   ] 
   (a) -- (b);
   \draw [
    postaction={
        decoration={contour lineto},
        draw=blue, ultra thick},
    postaction={
        decoration={contour lineto, contour distance=1.5pt}, 
        draw=red, decorate}
   ] 
   (a) -- (w);
   \draw [
    postaction={
        decoration={contour lineto},
        draw=blue, ultra thick},
    postaction={
        decoration={contour lineto, contour distance=1.5pt}, 
        draw=red, decorate}
   ] 
   (c) -- (b);
   \draw [
    postaction={
        decoration={contour lineto},
        draw=blue, ultra thick},
    postaction={
        decoration={contour lineto, contour distance=1.5pt}, 
        draw=red, decorate}
   ] 
   (c) -- (w);
  \end{tikzpicture}
\end{center}
\item[Case 3:] $G$ contains 
two incident double-colored edges $ab$ and $bc$, but no edge connecting $a$ and $c$. 
Let $V_1=\{a, b, c\}$, $V_2=V(G)\setminus V_1$.
Then there cannot be $5$ edges from any vertex $w\in V_2$ to $V_1$, 
otherwise,  $G$ is a graph either in Case 1 or in Case 2.  
Thus there are at most $4$ edges from any 
vertex in $V_2$ to $V_1$. 
\item[Case 4:] If $G$ is not the above three cases, then for any double-colored edge connecting $a$ and $b$, 
there are at most $2$ edges from any other vertex to $\{a, b\}$. 
\end{description}

Applying the inductive hypothesis to $G[V_2]$,
we have $$|E(G[V_2])|\leq \binom{|V_2|}{2}+ \left\lfloor \frac{|V_2|^2}{6}\right\rfloor.$$ 

Then the number of edges in $G$ is:
for the first three cases, 
\begin{align*} \label{eq:H_8a}
\begin{split}
|E(G)|
&= |E(G[V_1])| + |E(G[V_2])| +|E(V_1, V_2)|\\ 
&\leq 6 + \binom{n-3}{2}+ \left\lfloor \frac{(n-3)^2}{6}\right\rfloor + 4(n-3)\\
&=\binom{n}{2}+ 3+ \left\lfloor \frac{(n-3)^2}{6}\right\rfloor \\
&=\binom{n}{2}+\left\lfloor \frac{n^2-6n+27}{6}\right\rfloor\\
&\leq \binom{n}{2}+\left\lfloor \frac{n^2+3}{6}\right\rfloor\\
&= \binom{n}{2}+\left\lfloor \frac{n^2}{6}\right\rfloor, 
\end{split}
\end{align*}

for Case 4, 
\begin{align*}
|E(G)| &=|E(G[V_1])| + |E(G[V_2])| +|E(V_1, V_2)|\\
&\leq 2 + \binom{n-2}{2}+ \left\lfloor \frac{(n-2)^2}{6}\right\rfloor+ 2(n-3) \\
&= \binom{n}{2} +1+\left\lfloor \frac{(n-2)^2}{6}\right\rfloor \\
&= \binom{n}{2} +\left\lfloor \frac{n^2-4n+10}{6}\right\rfloor \\
&<\binom{n}{2} +\left\lfloor \frac{n^2}{6}\right\rfloor.
\end{align*}
The induction step is finished. 
It follows that $\pi(\{T_1, T_2\})\leq \frac{4}{3}$.  
\end{proof}

\begin{lemma}\label{H8}
$\pi(H_8)= \frac{4}{3}$.
\end{lemma}

\begin{proof}
We first prove $\pi(H_8)\leq \frac{4}{3}$. To show this, 
we prove that $H_8$ is $T_1$ and $T_2$-colorable,  
i.e there are graph homomorphisms from 
$H_8$ to $T_1$ and from 
$H_8$  to $T_2$. 

\begin{description}
\item[For $T_1$:] We define a map $f$ by $f(ACX)=f(BCX)=4, f(ADY)=f(BDY)=3$, 
$f(ACY)=f(BCY)=2, f(ADX)=f(BDX)=1$.
One can check that $f$ is a graph homomorphism from  $H_8$ to $T_1$.  
\item[For $T_2$:]  We define a map $g$ by $g(ACX)= g(ADX)=1, g(ADY)=g(ACY)=3, 
g(BDX)= g(BDY)=2, g(BCY)=g(BCX)=4$. It is easy to check that $g$ is a graph  homomorphism 
from  $H_8$ to $T_2$. 
\end{description}
For any positive integer $n$, let $G_n$ be a $2$-colored graph on $n$ vertices such that 
$h_n(G_n)\geq \pi(T_1, T_2)+\epsilon=\pi(T_1(s), T_2(s))+\epsilon$, for any $s\geq 2$ and $\epsilon>0$.
Then $G_n$ contains $T_1(s)$ or $T_2(s)$ as subgraph, further $G_n$ contains $H_8$ as subgraph.
Then $\pi(H_8)\leq \pi(\{T_1, T_2\})$. By Lemma \ref{T1T2free}, $\pi(H_8)\leq \frac{4}{3}$. 
By Remark \ref{4/3then1}, if $\pi(H_8)< \frac{4}{3}$, then $\pi(H_8)=1$, while $H_8$ is not $T$-colorable, a contradiction. Thus it must be the case $\pi(H_8)=\frac{4}{3}$.  
\end{proof}

\begin{remark}\label{3/2then4/3}
As we know, if $\pi(H)< \frac{3}{2}$, it must be colorable by $G_c$ and its variations, then it must be be colorable by  $H_8$ according to Lemma \ref{productcolorable}. Thus $\pi(H)\in \{1, \frac{4}{3}\}$.
\end{remark} 

For convenience, we use numbers to represent vertices: $ACX=1, ADY=2, ACY=3, ADX=4, BDX=5, BCY=6, BCX=7, BDY=8$. 
Then $H_8$ has edges:
 $$E_r(H_8)=\{12, 13 ,24,  34,  16, 37, 48 , 25, 35, 18, 46, 27\};$$
 $$E_b(H_8)=\{56, 57, 68, 78, 26, 15, 47, 38, 35, 18, 46, 27\}.$$  
 
Now we are ready to finish the proof of Theorem \ref{allresultsfor2colored}. 
\begin{proof}[Proof of Items 2 and 3 in Theorem \ref{allresultsfor2colored}.]
By Remark \ref{4/3then1}, Remark \ref{3/2then4/3} and Lemma \ref{3/2}, the Tur\'an densities of all bipartite $2$-colored graphs are in the set $\left\{1, \frac{4}{3}, \frac{3}{2}\right\}$. 
To show Item 2, let $H$ be a $2$-colored graph with $\pi(H)= \frac{4}{3}$, then $H$ must be $H_8$-colorable. One can check if $H$ does not contain $T$ as a sub-graph, then $H$ must be $T$-colorable, implying  $\pi(H)= 1$, a contradiction.
By excluding the bipartite $2$-colored graphs in Item 2, we obtain the result in Item 3. 
\end{proof}

\begin{eg}
  Let $T_3$ be the following 2-colored graph, $T_3$ is non-degenerate and $\pi(T_3)=\frac{4}{3}$.
\end{eg}
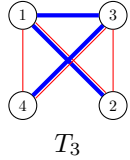
\begin{center}
\begin{tikzpicture} [scale=0.6,
    every edge quotes/.append style={font=\scriptsize\sffamily, midway, auto, inner sep=1pt,  double},
    vertex/.style={circle, draw=black, fill=white, scale=0.6}
  ]
  \foreach \j/\i in {(0, 2)/1,(2, 2)/3,(2,0)/2,(0, 0)/4}
  \node [vertex] (\i) at \j {\i};
  \draw  [color=red] (1) edge (4);
  \draw  [color=red] (2) edge (3);
  \draw  [color=blue,  ultra thick] (1) edge  (3);
  
\draw [
    postaction={
        decoration={contour lineto},
        draw=blue, ultra thick},
    postaction={
        decoration={contour lineto, contour distance=1.5pt}, 
        draw=red, decorate}
   ] 
   (1) -- (2);
   \draw [
    postaction={
        decoration={contour lineto},
        draw=blue, ultra thick},
    postaction={
        decoration={contour lineto, contour distance=1.5pt}, 
        draw=red, decorate}
   ] 
   (3) -- (4);
  \end{tikzpicture}
%
\\ $T_3$ 
\end{center}

\section{The degenerate \{2, 3\}-graphs}\label{23}

%
%

In this section, we study the degenerate $\{2, 3\}$-graphs and show an application of the study of $2$-edge-colored graphs on Tur\'an density of $\{2, 3\}$-graphs. A $\{2, 3\}$-graph is a non-uniform hypergraph where each edge consists of $2$ or $3$ vertices. Given a $\{2, 3\}$-graph $G$, we call an edge of cardinality $i$ as an $i$-edge,  
and use $E_i(G)$ to represent the set of $i$-edges. Thus $G$ can be represented by $G=(V(G), E_2(G), E_3(G))$. 
A $2$-edge $e$ is called a double edge if  $e\subset f$, for some $3$-edge $f\in E_3(G)$.
For convenience, we use the form of $ac$ to denote the edge $\{a, b\}$ and use $abc$ to denote the edge $\{a, b, c\}$. The notation $H_{n}^{\{2, 3\}}$ represents a $\{2, 3\}$-graph on $n$ 
vertices, $K^{\{2, 3\}}_n$ represents the complete hypergraph on $n$ vertices with edge set 
$\binom{[n]}{2}\cup \binom{[n]}{3}$.

Given a family of $\{2, 3\}$-graphs $\mathcal H$,
the Tur\'an density of  $\mathcal H$ is defined to be:
\begin{align*}
\pi(\mathcal H)& =\lim_{n\to \infty}\pi_n(\mathcal H)=
\lim_{n \to \infty} {\max}\left\{\frac{|E_2(G)|}{\binom{n}{2}} +  \frac{|E_3(G)|}{\binom{n}{3}}\right\}, 
\end{align*}
where the maximum is taken over all $H$-free hypergraphs $G$ on $n$ vertices satisfying
$G\subseteq K^{\{2, 3\}}_n$, and $G$ is $\mathcal{H}$-free $\{2, 3\}$-graph. 
Please refer to \cite{LuBai} for details on the Tur\'an density of non-uniform hypergraphs.

Next let us see some definitions and results for $\{2, 3\}$-graphs. 

\begin{defn}\cite{JLU} 
Let $H$ be a hypergraph containing some $2$-edges. The $2$-subdivision of
$H$ is a new hypergraph $H'$ obtained from $H$ by subdividing each $2$-edge simultaneously.
Namely, if $H$ contains $t$ $2$-edges, add $t$ new vertices $x_1, \ldots, x_t$ to $H$ and for $i=1, 2, \ldots, t$ and  replace the $2$-edge $\{u_i, v_i\}$ with $\{u_i, x_i\}$ and $\{x_i, v_i\}$. 
\end{defn}
\begin{theorem}\cite{JLU} \label{subdivision}
Let $H'$ be the $2$-subdivision of $H$. If $H$ is degenerate, then so is $H'$.
\end{theorem}

\begin{defn}\cite{JLU}
The suspension of a hypergraph $H$, denoted by $S(H)$, is the hypergraph
with $V = V (H)\cup \{v\}$  where $\{v\}$ is a new vertex not in $V(H)$, and the edge set
$E=\{e\cup \{v\}: e\in E(H)\}$.  We write $S^t(H)$ to denote the hypergraph obtained by iterating
the suspension operation $t$-times, i.e. $S^2(H) = S(S(H))$ and $S^3(H) = S(S(S(H)))$,
etc.
\end{defn}

\begin{prop}\cite{JLU}\label{suspenprop}
For any family of hypergraphs $\mathcal{H}$ we have that $\pi(S(\mathcal{H}))\leq \pi(\mathcal{H})$.
\end{prop}

\begin{theorem}\cite{LuBai} \label{anyRnontrivial}
  Let $R$ be a set of distinct positive integers with $|R|\geq 2$ and $R\neq \{1, 2\}$. Then a
  non-trivial degenerate $R$-graph always exists. 
\end{theorem}

In \cite{LuBai}, we say a degenerate $R$-graph is {\em trivial} if it is a sub-graph of a blow-up of the chain $C^{R}$. By Theorem \ref{anyRnontrivial}, there exist non-trivial degenerate $\{2, 3\}$-graphs.  The $\{2, 3\}$-graph $H=\{12, 123\}$ is a chain, thus it is degenerate. By Theorem \ref{subdivision}, the subdivision $H'=\{14, 24, 123\}$ is also degenerate, but it is non-trivial. As showed in \cite{JLU}, 
$H^0=S(K_2^{1, 2})=\lbrace 13, 12, 123\rbrace$ is not degenerate, and 
$\pi(H^0)=\frac{5}{4}$. 

So what does the degenerate $\{2, 3\}$-graph look like?  To answer this question, we may need to
construct a family of $\{2, 3\}$-graphs $G_n$ with $h_n(G_n)>(1+\epsilon)$ for some $\epsilon>0$. 
Here are three $\{2, 3\}$-graphs with edge density greater than $1$.

Note that for any $R$-graph $H$ (with possible loops), one can construct the family of
$H$-colorable $R$-graph by blowing up $H$ in certain way. 
The langrangian of $H$ is the maximum edge density of the $H$-colorable $R$-graph
that one can get this way. For more details 
of $R$-graphs with loops, blow-up, and Lagrangian, please refer to \cite{LuBai}. 
In this part, we will use an easy-understood way to calculate the edge densities. 
\begin{eg}
A $\{2, 3\}$-graph $G_1^{\{2, 3\}}$ is a blowing-up of the general hypergraph $H_1$ with vertex set $\{a, b, c\}$ 
and edge set $\{aa, bb, cc, abc\}$, if there exists a partition of vertex set 
such that $V(G_1^{\{2, 3\}})=A\cup B \cup C$ and every $2$-edge meets two vertices in $A$ (or $B$, or $C$), every $3$-edge meets $A, B, C$ one vertex respectively. In other words, 
$$E(G_1^{\{2, 3\}})=\binom{A}{ 2} \cup  \binom{A}{ 1}\binom{B}{ 1} \cup  \binom{A}{ 1}\binom{C}{ 1} 
\cup  \binom{A}{1}\binom{B}{ 1}\binom{C}{1}.$$

Let $|A|=xn$ and $|B|=|C|=\frac{1-x}{2}n$ for some value $x\in (0,1)$.
We have 
\begin{align*}
h_n(G_1^{\{2, 3\}})
& =\frac{ \binom{xn}{2}+\binom{xn}{1}\binom{(1-x)n}{1} }{\binom{n}{2}}+
\frac{xn(\frac{(1-x)n}{2})^2}{{n \choose 3}} \\
&= x^2 +2x(1-x)+\frac{3}{2}x(1-x)^2+o_n(1)\\
&=\frac{7}{2}x-4x^2+\frac{3}{2}x^3+o_n(1).
\end{align*}
The above value reaches the maximum value $\frac{245}{243}+o_n(1)$ at $x=\frac{7}{9}$.
 
 

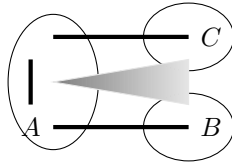
\begin{center}
  \begin{tikzpicture}[scale=0.3]
    \draw (0,0) ellipse (2 and 3);
    \draw (6,2) ellipse (2 and 1.5);
    \draw (6,-2) ellipse (2 and 1.5);
\draw[draw=black, ultra thick] (-1, -1) -- (-1, 1);
\draw[draw=black, ultra thick] ( 0, 2) -- (6, 2);
\draw[draw=black, ultra thick] ( 0, -2) -- (6, -2);

\shadedraw[left color=black!50!white, right color=black!10!white, draw=black!10!white] (0,0)--(6,1)--(6,-1)--cycle;
\node at (-1,-2) {$A$};
\node at (7,-2) {$B$};
\node at (7,2) {$C$};

  \end{tikzpicture}
\\
$G_1^{\{2, 3\}}: \  h_n(G_1^{\{2, 3\}})=\frac{245}{243}$ at $|A|=\frac{7}{9}n$.
\end{center}

\end{eg}

\begin{eg}
A $\{2, 3\}$-graph $G_2^{\{2, 3\}}$ is  a blowing-up of   the general hypergraph $H_2$ with vertex set $\{x, y\}$ 
and edge set $\{xy, xxx, xxy\}$, if there exists a partition of vertex set 
such that $V(G_2^{\{2, 3\}})=X\cup Y$ and  every $2$-edge meets one vertex in $X$ and one vertex in $Y$, 
every $3$-edge either meet three vertices in $X$ or two vertices in $X$ plus one vertex in $Y$. Actually 
$G_2^{\{2, 3\}}$ is $H_2$-colorable.
In other words, 
$$E(G_2^{\{2, 3\}})=\binom{X}{ 3} \cup  \binom{X}{ 2}\binom{Y}{ 1} 
\cup  \binom{X}{1}\binom{Y}{ 1}.$$

Let $|X|=xn$ and $|Y|=(1-x)n$ for some value $x\in (0,1)$, 
we have 
\begin{align*}
h_n(G_2^{\{2, 3\}})
& =\frac{ \binom{xn}{3}+\binom{xn}{2}\binom{(1-x)n}{1} }{\binom{n}{3}}+\frac{
xn(1-x)n}{{n \choose 2}} \\
&= x^3 +3x^2(1-x)+2x(1-x)+o_n(1)\\
&=2x+x^2-2x^3+o_n(1).
\end{align*}
The above value reaches the maximum value $\frac{19+13\sqrt{13}}{54}+o_n(1)
\approx 1.21985...+o_n(1)$ at $x=\frac{1+\sqrt{13}}{6}$.

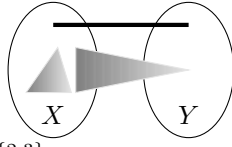
\begin{center}
  \begin{tikzpicture}[scale=0.3]
    \draw (0,0) ellipse (2 and 3);
    \draw (6,0) ellipse (2 and 3);

\draw[draw=black, ultra thick] ( 0, 2) -- (6, 2);

\shadedraw[left color=black!50!white, right color=black!10!white, draw=black!10!white] (-1.2,-1)--(0.8,-1)--(0.2,1)--cycle;
\shadedraw[left color=black!50!white, right color=black!10!white, draw=black!10!white] (1,-1)--(1,1)--(6,0)--cycle;

\node at (0,-2) {$X$};
\node at (6,-2) {$Y$};

  \end{tikzpicture}
\\
$G_2^{\{2, 3\}}:\  h_n(G_2^{\{2, 3\}})\approx 1.21985$ at $|X|=(\frac{1+\sqrt{13}}{6})n$.
\end{center}

\end{eg}
\begin{eg}

A $\{2, 3\}$-graph $G_3^{\{2, 3\}}$ is  a blowing-up of the general hypergraph $H_3$ with vertex set $\{e, f\}$ 
and edge set $\{ee, eef\}$, if there exists a partition of vertex set 
such that $V(G_2^{\{2, 3\}})=E\cup F$ and  every $2$-edge meets two vertices in $E$, 
every $3$-edge meets two vertices in $E$ plus one vertex in $\textsc{F}$. Actually 
$G_3^{\{2, 3\}}$ is $H_3$-colorable.
In other words, 
$$E(G_3^{\{2, 3\}})=\binom{E}{2} \cup  \binom{E}{2}\binom{Y}{1}.$$

Let $|E|=xn$ and $|F|=(1-x)n$ for some value $x\in (0,1)$, 
we have 
\begin{align*}
h_n(G_3^{\{2, 3\}})
& =\frac{ \binom{xn}{2}}{\binom{n}{2}}+
\frac{\binom{xn}{2}\binom{(1-x)n}{1}}{{n \choose 3}} \\
&= x^2 +3x^2(1-x)+o_n(1)\\
&=4x^2-3x^3+o_n(1).
\end{align*}
The above value reaches the maximum value $\frac{256}{243}+o_n(1)$ at $x=\frac{8}{9}$.

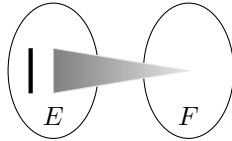
\begin{center}

  \begin{tikzpicture}[scale=0.3]
    \draw (0,0) ellipse (2 and 3);
    \draw (6,0) ellipse (2 and 3);
\draw[draw=black, ultra thick] (-1, -1) -- (-1, 1);
\shadedraw[left color=black!50!white, right color=black!10!white, draw=black!10!white] (0,-1)--(0,1)--(6,0)--cycle;
\node at (0,-2) {$E$};
\node at (6,-2) {$F$};
  \end{tikzpicture}
\\
$G_3^{\{2, 3\}}:\  h_n(G_3^{\{2, 3\}})=\frac{256}{243}$ at $|E|=\frac{8}{9}n$.

\end{center}

\end{eg}

A degenerate $\{2, 3\}$-graph  must appear as sub-graphs in all above
$\{2, 3\}$-graphs $G_1^{\{2, 3\}}, G_2^{\{2, 3\}}$ and $G_3^{\{2, 3\}}$,  
thus it must appear as sub-graph in the product of these hypergraphs. 
By taking this product, we get a $12$-vertex  $\{2, 3\}$-graph which is $H_{9}^{\{2, 3\}}$-colorable. Thus we have

\begin{lemma}
The degenerate $\{2, 3\}$-graphs must be $H_{9}^{\{2, 3\}}$-colorable.  
\end{lemma}
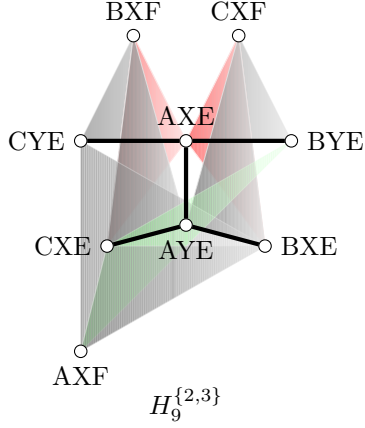
\begin{center}
\begin{tikzpicture}[scale=0.7,vertex/.style={circle, draw=black, fill=white}]
\shadedraw[left color=black!50!white, right color=black!10!white, draw=black!10!white, opacity=0.5] (-1,2)--(-2,0)--(0,0)--cycle;
\shadedraw[left color=red!50!white, right color=red!10!white, draw=red!10!white]
 (-1,2)--(-1.5,-2)--(0,0)--cycle;
\shadedraw[left color=black!50!white, right color=black!10!white, draw=black!10!white, opacity=0.8]
 (-1,2)--(-1.5,-2)--(0,-1.6)--cycle;
\shadedraw[left color=black!50!white, right color=black!10!white, draw=black!10!white, opacity=0.5]
 (1,2)--(2,0)--(0,0)--cycle;
\shadedraw[left color=red!50!white, right color=red!10!white, draw=red!10!white] 
(1,2)--(1.5,-2)--(0,0)--cycle;
\shadedraw[left color=black!50!white, right color=black!10!white, draw=black!10!white, opacity=0.8] 
(1,2)--(1.5,-2)--(0,-1.6)--cycle;
\shadedraw[left color=black!50!white, right color=black!10!white, draw=black!10!white, opacity=0.4] 
(-2,-4)--(-2,0)--(1.5,-2)--cycle;
\shadedraw[left color=green!50!white, right color=green!10!white, draw=green!10!white, opacity=0.3]
(-2,-4)--(2,0)--(-1.5,-2)--cycle;
\shadedraw[left color=black!50!white, right color=black!10!white, draw=black!10!white, opacity=0.4]
(-2,-4)--(-1.5,-2)--(1.5,-2)--cycle;

\node at (0,0) [vertex, scale=0.5] (v1) [label=above:{AXE}] {};
\node at (-2, 0) [vertex, scale=0.5] (v2) [label=left:{CYE}] {};
\node at (2, 0) [vertex, scale=0.5] (v3) [label=right:{BYE}] {};
\node at (-1,2) [vertex, scale=0.5] (v4) [label=above:{BXF}] {};
\node at (1,2) [vertex, scale=0.5] (v5) [label=above:{CXF}] {};
\node at (-1.5,-2) [vertex, scale=0.5] (v8) [label=left:{CXE}] {};
\node at (1.5,-2) [vertex, scale=0.5] (v9) [label=right:{BXE}] {};
\node at (-2,-4) [vertex, scale=0.5] (v10) [label=below:{AXF}] {};
\node at (0,-1.6) [vertex, scale=0.5] (v12) [label=below:{AYE}] {};
\draw [draw=black, ultra thick] (v1) -- (v2);
\draw [draw=black, ultra thick] (v1) -- (v3);
\draw [draw=black, ultra thick] (v1) -- (v12);
\draw [draw=black, ultra thick] (v8) -- (v12);
\draw [draw=black, ultra thick] (v9) -- (v12);
\node at (0,-5) {$H_{9}^{\{2, 3\}}$} ;
\end{tikzpicture}
\end{center}

 The following lemma shows a relation between such $\{2, 3\}$-graphs and the $2$-colored graphs and can help us  determine the upper bound for the Tur\'an density of some  $\{2, 3\}$-graphs. 

\begin{theorem}\label{relation22and23}
Let $H=(V, \ E_r, \ E_b)$ be a $2$-colored graph, and $H'=(V', E_2, \ E_3)$ be a $\{2, 3\}$-graph obtained from $H$ by adding a new vertex $v\not\in (V)$ such that $V'=V\cup \{v\}$ and $E_2=E_r$, and $E_3=\{e' | e'=e\cup{v}, e\in E_b\}$. Then $\pi(H')\leq \pi(H)$. 
\end{theorem}

\begin{proof}
Let $n$ be positive integer, let $G=(V, \ E_2(G), \ E_3(G))$ be an arbitrary $H'$-free $\{2, 3\}$-graph on $n$ vertices. 
For any vertex $v\in V(G)$, let $G_v=(V(G)\setminus \{v\}, E_{v,2},  E_{v,3})$ be a $2$-colored graph obtained form $G$,  such that the red edges are $E_{v,2}=E_2(G)$, the blue edges are $E_{v,3}=\{{u, w} | \{vuw\}\in E_3\}$. 
Observe that $G_v$ is $H$-free since $G$ is $H'$-free. Thus $h_{n-1}(G_v)\leq \pi_n(H).$

Since $$|E_2(G)|= \frac{1}{n-2}\sum\limits_{v\in V(G)} |E_{v,2}| \  \text{and }\ |E_3(G)|=\frac{1}{3}\sum\limits_{v\in V(G)} |E_{v,3}|,$$
Then 
\begin{align*}
h_n(G)&=\frac{|E_2(G)|}{{n\choose 2}} + \frac{|E_3(G)|}{{n\choose 3}}\\
&=\sum\limits_{v\in V(G)} \frac{|E_{v,2}|}{(n-2){n\choose 2}} + 
\sum\limits_{v\in V(G)} \frac{|E_{v,3}|}{3{n\choose 3}}\\
&=\frac{1}{n}\sum\limits_{v\in V(G)} \frac{|E_{v,2}|}{{n-1\choose 2}} +
\frac{1}{n}\sum\limits_{v\in V(G)}  \frac{|E_{v,3}|}{{n-1\choose 2}}\\
& =\frac{1}{n}\sum\limits_{v\in V(G)} \left(  \frac{|E_{v,2}|}{{n-1\choose 2}}+ \frac{|E_{v,3}|}{{n-1\choose 2}}\right)\\
&\leq \frac{1}{n}\sum\limits_{v\in V(G)} h_{n-1}(G_v) \\
&\leq  \pi(H).
\end{align*}
Therefore $\pi(H')\leq \pi(H)$. 
\end{proof}

So far we couldn't give an upper bound of $\pi(H_{9}^{\{2, 3\}})$, but we can show a sub-graph of 
$\pi(H_{9}^{\{2, 3\}})$ are degenerate using above theorem. 
Let us observe that if we remove a single vertex $AXF$ and edges connecting to it, the resulting sub-hypergraph is $H_{5}^{\{2, 3\}}$-colorable, 
where $H_{5}^{\{2, 3\}}=([5], \{12, 13, 34, 125, 135, 345\}).$
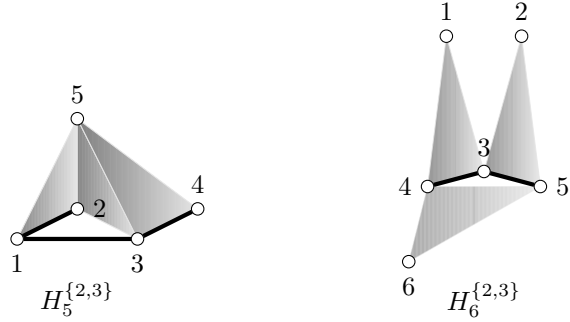
\begin{center}
\begin{tikzpicture}[scale=0.8, vertex/.style={circle, draw=black, fill=white}]
\shadedraw[left color=black!50!white, right color=black!10!white, draw=black!10!white] (0,0)--(1,0.5)--(1,2)--cycle;
\shadedraw[left color=black!50!white, right color=black!10!white, draw=black!10!white] (2,0)--(1,0.5)--(1,2)--cycle;
\shadedraw[left color=black!50!white, right color=black!10!white, draw=black!10!white] (2,0)--(3,0.5)--(1,2)--cycle;

\node at (0,0) [vertex, scale=0.5] (v1) [label=below:{1}] {};
\node at (1,0.5) [vertex, scale=0.5] (v2) [label=right:{2}] {};
\node at (2,0) [vertex, scale=0.5] (v3) [label=below:{3}] {};
\node at (3,0.5) [vertex, scale=0.5] (v4) [label=above:{4}] {};
\node at (1,2) [vertex, scale=0.5] (v5) [label=above:{5}] {};
\draw [draw=black, ultra thick] (v1) -- (v2);
\draw [draw=black, ultra thick] (v1) -- (v3);
\draw [draw=black, ultra thick] (v3) -- (v4);
\node at (1,-1) {$H_{5}^{\{2, 3\}}$} ;
\end{tikzpicture}
\hfil
\begin{tikzpicture}[scale=0.5, vertex/.style={circle, draw=black, fill=white}]

\shadedraw[left color=black!50!white, right color=black!10!white, draw=black!10!white, opacity=0.8]
 (-1,2)--(-1.5,-2)--(0,-1.6)--cycle;
\shadedraw[left color=black!50!white, right color=black!10!white, draw=black!10!white, opacity=0.8] 
(1,2)--(1.5,-2)--(0,-1.6)--cycle;
\shadedraw[left color=black!50!white, right color=black!10!white, draw=black!10!white, opacity=0.5] 
(-2,-4)--(-1.5,-2)--(1.5,-2)--cycle;

\node at (-1,2) [vertex, scale=0.5] (v4) [label=above:{1}] {};
\node at (1,2) [vertex, scale=0.5] (v5) [label=above:{2}] {};

\node at (-1.5,-2) [vertex, scale=0.5] (v8) [label=left:{4}] {};
\node at (1.5,-2) [vertex, scale=0.5] (v9) [label=right:{5}] {};

\node at (-2,-4) [vertex, scale=0.5] (v10) [label=below:{6}] {};
\node at (0,-1.6) [vertex, scale=0.5] (v12) [label=above:{3}] {};
\draw [draw=black, ultra thick] (v8) -- (v12);
\draw [draw=black, ultra thick] (v9) -- (v12);
\node at (0,-5) {$H_{6}^{\{2, 3\}}$} ;
\end{tikzpicture}
\end{center}
Observe that we can also obtain $H_{5}^{\{2, 3\}}$ from $T$ by adding vertex $5$, and connect it with blue edges. Thus we have $\pi(H_{5}^{\{2, 3\}})=1$.

In $H_{9}^{\{2, 3\}}$, removing a single $2$-edge connecting vertices $AXE$ and $AYE$, the resulting sub-graph is $H_{6}^{\{2, 3\}}$-colorable,
where  $H_{6}^{\{2, 3\}}=([6], \{34, 35, 134, 235, 456\}).$ 
However, we don't know the Tur\'an density of $H_{6}^{\{2, 3\}}$. 
We remark that determining the degenerate $\{2, 3\}$-hypergraph is still unknown.



\begin{thebibliography}{1}
\bibliographystyle{plain}
\bibliography{bibfile}
\bibitem{diwanmubayi} Ajit. Diwan and D. Mubayi . Turan's theorem with colors. (2006).  

\bibitem{AMB}  N. Alon, C. McDiarmid and B. Read, Acyclic colorings of graphs, {\it Random
Structures and Algorithms} 2 (1991), 277-289.


\bibitem{AK}  N. Alon,  M. Krivelevich and B. Sudakov, Tur\'an Numbers of Bipartite Graphs and Related Ramsey-Type Questions, {\it Combinatorics, Probability and Computing } 12,(2003), 477-494.

\bibitem{LuBai}S. Bai and L. Lu, On the Tur\'an density of $\{1, 3\}$-Hypergraphs, {\em Published to The electronic journal of combinatorics 26(1), (2019)}. 

\bibitem{MZ} D. Mubayi and Y. Zhao, Non-uniform Tur\'an-type problems, {\em J. Comb. Th. A}, 111 (2004), 106-110.


\bibitem{EGL} P. Erd\H{o}s, A. Gy\'arf\'as and L. Pyber, Vertex coverings by monochromatic
cycles and trees,{\it J. Combin. Theory, Ser. B} 51 (1991), 90-95.

\bibitem{ES} P. Erd\H{o}s and A. Stone, On the structure of linear graphs, {\em Bull. Am. Math.
Soc.} 52 (1946), 1087-1091.

\bibitem{ESS1}  P. Erd\H{o}s and M. Simonovits,  A limit theorem in graph theory. {\it Studia Scientiarum Mathematicarum Hungarica, } 1(1966). 


\bibitem{ESS2}  P. Erd\H{o}s, On the structure of linear graphs. {\it Israel Journal of Mathematics}, 1(1963), 156-160.  


\bibitem{Turan} P. Tur\'{a}n. On an extremal problem in graph theory (in Hungarian). Matematikai\`{e}es Fizikai Lapok 48(1941), 436–452.

\bibitem{JLU} T. Johnston and  L. Lu,  Tur\'{a}n problems on non-uniform hypergraphs,
 {\em Electronic Journal of Combinatorics,} 21(4) (2014),  22-56.

\bibitem{KNS} Gy Katona, T. Nemetz and M. Simonovits, On a problem of Tur\'an in the theory of graphs, {\it Mat. Lapok} 15 (1964) 228-238.




\bibitem {PHTZ} Y. Peng, H. Peng, Q. Tang and C. Zhao, An extension of the Motzkin–Straus theorem to non-uniform hypergraphs and its applications, {\it Discrete Applied Mathematics}, 200 (2016),  170-175. 


\bibitem{Kee} P. Keevash, Hypergraph Tur\'an problems, {\it in Surveys in Combinatorics, Cambridge University Press, Cambridge,} 2011,  83-140.





\bibitem{Mental} W. Mantel. Problem 28, Solution by H. Gouwentak, W. Mantel, J. Teixeira de Mattes, F. Schuh and W. A. Wythoff. {\it Wiskundige Opgaven}, (1907), 60-61. 

 
\bibitem{dege} Zolt\'{a}n. F\"{u}redi  and Mikl\'{o}s. Simonovits.  The History of Degenerate (Bipartite) Extremal Graph Problems. {\it Bolyai Soc. Math. Stud.. 25. 10.1007/978-3-642-39286-3-7. }(2013). 


 
\end{thebibliography}
\end{document}